\documentclass[11pt]{scrartcl}

\usepackage{amsmath}
\usepackage{amssymb}

\newcommand{\N}{\mathbb{N}}
\newcommand{\R}{\mathbb{R}}
\renewcommand{\H}{\R^d}

\DeclareMathOperator*{\argmin}{arg\,min}

\newcommand{\prox}[3][]{\operatorname{prox}^{#1}_{#2}\left(#3 \right)}

\usepackage{amsthm}
\theoremstyle{plain}
\newtheorem{theorem}{Theorem}[section]
\newtheorem{proposition}[theorem]{Proposition}
\newtheorem{lemma}[theorem]{Lemma}
\newtheorem{corollary}[theorem]{Corollary}

\theoremstyle{definition}
\newtheorem{definition}[theorem]{Definition}
\newtheorem{algo}[theorem]{Algorithm}
\theoremstyle{remark}

\usepackage[sort,nocompress]{cite}

\usepackage{amsfonts}

\usepackage{enumerate}

\usepackage{mathtools}

\newcommand{\Id}{\textup{Id}}

\usepackage{mathtools}
\begin{document}

	\begin{center}
		{\LARGE Ubiquitous algorithms in convex optimization\smallskip \\ generate self-contracted sequences}
	\end{center}

	\medskip

	\begin{center}
		{\large \textsc{Axel B\"ohm, Aris Daniilidis}}
	\end{center}

\bigskip
\noindent \textbf{Abstract.} In this work we show that various algorithms, ubiquitous in convex optimization (e.g. proximal-gradient, alternating projections and averaged projections) generate self-contracted sequences $\{x_{k}\}_{k\in\N}$. As a consequence, a novel universal bound for the \emph{length} ($\sum_{k\ge 0}\Vert x_{k+1}-x_k\Vert$) can be deduced. In addition, this bound is independent of both the concrete data of the problem (sets, functions) as well as the stepsize involved, and only depends on the dimension of the space.

	\vspace{0.55cm}


	\vspace{0.55cm}


\section{Introduction}
\label{sec:introduction}

The notion of \textit{self-contracted curve} captures, under a simple metric definition (see forthcoming Definition~\ref{def:self_contracted}), characteristic properties of the gradient flow of a convex function, relevant for convergence. The term appeared for the first time in~\cite{DLS2010}, where it was shown that \emph{planar} self-contracted curves are rectifiable. Later on, exploring an old geometrical idea of Manselli-Pucci (see~\cite{MP1991}), the previous result has been extended to any finite dimensional Euclidean space. In particular, in~\cite{DDDL2015} (and independently in \cite{LMV2015} assuming continuity of the curves) it was shown that the length of any self-contracted curve in $\mathbb{R}^d$ is controlled by a universal constant $C_d$ (depending only on the dimension of the space) times the diameter of the image of the curve.\smallskip

The aforementioned control of the length directly yields uniform estimates for the asymptotic behaviour of the bounded orbits of quasiconvex gradient systems, convex subgradient systems as well as of the bounded orbits of convex foliations --- all of them being typical instances of self-contracted curves. Self-contractedness is indeed strongly related to convexity. It was shown in~\cite{DL2018} that under mild assumptions every smooth self-contracted curve can be obtained as an orbit of some smooth convex function. \smallskip

Another important feature of the notion of self-contractedness is that it translates naturally to the discrete case, to include sequences $\{x_k\}_{k\in\N}$ generated by some algorithmic scheme. (The series $\sum_{k\geq 0}\Vert x_{k+1}-x_k\Vert$ corresponds to the \textit{length of the sequence}.) A typical example consists of the iterates generated by the proximal-point algorithm applied to a convex function. These iterations, being obtained as successive projections to the convex foliation given by the sublevel sets of the function, generate a self-contracted sequence (cf.~\cite{DDDL2015}). This provides an independent proof of the convergence of the proximal-point algorithm. \smallskip

The objective of this work is to show that other classical iterative schemes such as the gradient descent algorithm of a smooth convex function with Lipschitz gradient, the alternating projection algorithm for two closed convex sets and the average projection method for finitely many closed convex sets, also generate self-contracted sequences. Consequently, a prior universal estimate for the convergence of all of these methods can be deduced. This estimate neither depends on the specific function nor on the choice of proximal parameters, since all self-contracted sequences/curves lying in the given bounded set admit a universal bound for their length. \smallskip

Our approach relies strongly on interpreting the aforementioned algorithms as particular instances of the proximal-gradient method (Forward-Backward algorithm), see Algorithm~\ref{alg:proxgrad} and then establishes that the iterates of the latter give a self-contracted sequence, see Theorem~\ref{thm:proxgrad_selfcontracted}. \smallskip

The proof of this central result is surprisingly simple, making astute use of an additional quadratic decay stemming from the strong convexity that appears in the proximal operator. This being said, establishing directly self-contractedness for the alternating projection algorithm is not an easy task, and might be quite involved even in the particular case that one of the convex sets is in fact a convex cone. Indeed, the generated sequence of this algorithm (and in general of all of the aforementioned algorithms) cannot be obtained, in any obvious way, via successive projections to some convex foliation related to our data. The only exception is the fixed-step gradient descent algorithm of a $C^{1,1}$-convex function (which, being identified with the proximal-point algorithm of another convex function, it can indeed be obtained with successive projections to some convex foliation). Therefore, overall, this new simple approach gives a technique for establishing self-contractedness, without passing through a convex foliation, which up-to-now was the only known way to proceed. In particular, as a by-product, we obtain a new proof for establishing self-contractedness of the proximal-point algorithm (cf. Corollary~\ref{cor:prox}).  \smallskip

Let us finally mention, for completeness, that self-contracted curves have also been considered in more general settings, emancipating from direct applications to asymptotic theory of dynamical systems or optimization algorithms. To this end, self-contracted curves have been studied in~\cite{DDDR2018} in Riemann manifolds, where rectifiability has been established via an involved proof that borrows heavily from the underlying Euclidean structure. Remarkably enough, recent works on the topic reveal that Euclidean structure is not a real restriction: generalizing the results of~\cite{L2016}, the authors in~\cite{ST2017} established that any self-contracted curve in any finite dimensional (potentianlly asymmetric) normed space is rectifiable. Futher extensions include CAT(0) spaces~\cite{O2020} and spaces with weak lower curvature bound~\cite{LOZ2019}. In view of these developments, it is possible that the notion of self-contracted curve will turn out to be relevant also for abstract dynamics in a metric setting (see \cite{AGS2008} e.g.)

\section{Preliminaries}%
\label{sec:preliminaries}

Throughout this paper, $\R^d$ will denote the $d$-dimensional Euclidean space and $\langle \cdot, \cdot\rangle$ its inner product which generates the distance $d(x,y):= \lVert x - y \rVert$. For a nonempty subset~$A$ we denote its \emph{diameter} by $\textup{diam}(A):= \sup\{ d(x,y) : x,y \in A \}$.

\begin{definition}[self-contracted curve]
  \label{def:self_contracted}
  Given a possibly unbounded interval $I\subset \R$, a map $\gamma: I \to \R^d$ is called \emph{self-contracted}, if for all $t_1,t_2,t_3 \in I$ such that $t_1 \le t_2 \le t_3$
  \begin{equation*}
    d(\gamma(t_3), \gamma(t_2)) \le d(\gamma(t_3), \gamma(t_1)).
  \end{equation*}
\end{definition}
Note that although originally inspired by continuous curves, this definition does not require any form of continuity or smoothness for the curve $\gamma$. In particular, taking $\gamma$ to be constant on each interval $[n,n+1)$, for all $n\in\N$, the definition also covers the case of discrete sequences. Formalizing this, we call a sequence $\{x_{k}\}_{k\in\N}$ in $\R^d$ self-contracted if for all $k_1, k_2, k_3 \in \N$ such that $k_1 \le k_2 \le k_3$
  \begin{equation*}
    d(x_{k_3}, x_{k_2}) \le d(x_{k_3},x_{k_1}).
  \end{equation*}

  This seemingly innocent property of self-contractedness has remarkable consequences. It was proven in~\cite[Theorem~3.3]{DDDL2015} that every self-contracted curve in a finite dimensional Euclidean space is rectifiable and its length satisfies
  $$\ell(\gamma) \le C_d\, \textup{diam}(\gamma(I)),$$
  where $C_d$ denotes a constant only depending on the dimension of the space.
  Therefore, any bounded self-contracted sequence $\{x_{k}\}_{k\in\N}$ converges to some $x_{\infty}$ and
  \begin{equation}
	\label{eq:bound-sc}
    \sum_{k=1}^{\infty} d(x_{k+1}, x_{k} ) \le C_d \, d(x_0, x_{\infty}).
  \end{equation}
  The aim of this work is to establish the self-contractedness of several classical algorithms in convex optimization. Previous convergence proofs relied on specific Lyapunov functions, in particular, the characteristic property of Fejer monotonicity with respect to the solution set $\mathcal{S}$ (cf.~\cite[Definition~5.1]{bc}), meaning $d_{\mathcal{S}}(x_{k+1}) \le d_{\mathcal{S}}(x_{k})$. Making use of the additional information that the iterates form a self-contracted sequence, we obtain a \textit{data independent} bound given by~\eqref{eq:bound-sc}. This bound can be further improved, using Fejer monotonicity, to
  \begin{equation*}
    \sum_{k=1}^{\infty} d(x_{k+1}, x_{k} ) \le C_d \, d_{\mathcal{S}}(x_0),
  \end{equation*}
  whenever $x_{\infty}\in \mathcal{S}$ (which can always be ensured in the forthcoming algorithm).

\section{Proximal-gradient generates self-contracted iterates}%
\label{sec:proximal_gradient_is_self_contracted}

Consider the \emph{classical} problem
\begin{equation}
  \label{eq:convex_splitting}
  \min_{x \in \H{}}\, g(x) + f(x)
\end{equation}
for a proper, convex and lower semicontinuous function $g:\H\to\overline{\R}$ and a differentiable convex function $f:\H\to\R$ with $L$-Lipschitz continuous gradient. We associate with the above system the \emph{Forward-Backward} or \emph{Proximal-Gradient} (cf.~\cite[Section 27.3]{bc}) operator
\begin{equation}
  \label{eq:prox-grad-op}
  T_{\alpha}(x) := \prox{\alpha g}{x - \alpha \nabla f(x)}
\end{equation}
with stepsize $\alpha>0$.

\subsection{Stepsize bounded by the inverse of the Lipschitz constant}%

The most established method to solve the above problem is described below:

\begin{algo}[Proximal-Gradient-Method]%
  \label{alg:proxgrad}
  In the above setting, for $x_{0} \in \R^{d}$ and a sequence of stepsizes $\{\alpha_{k}\}_{k\in\N} \subseteq (0, 1/L)$, consider the following iterative scheme
  \begin{equation*}
     \quad x_{k+1} = T_{\alpha_{k}}(x_{k}), \quad \forall k \ge 0.
  \end{equation*}
\end{algo}

\begin{theorem}[Main result]
  \label{thm:proxgrad_selfcontracted}
  The iterates generated by Algorithm~\ref{alg:proxgrad} (Proximal-Gradient-Method with variable stepsize) form a self-contracted sequence.
\end{theorem}
For the proof we shall make use of the following three lemmata. The first two are well known and will be quoted without proof. The third lemma is also quite standard for these problems. \smallskip

Before we proceed, let us first recall that the \emph{subdifferential $\partial \Phi$} of a convex function $\Phi:\R^{d} \to \overline{\R}$ at $x$ is defined as follows
\begin{equation*}
  \partial \Phi(x) := \{p \in \R^{d}: \, \Phi(x) + \langle p, y -x \rangle \le \Phi(y),\, \forall y \in \R^{d}\}.
\end{equation*}
In particular, a point $x^{*}\in\R^d$ is a minimizer of $\Phi$ ($x^{*}\!\in\!\argmin{\Phi}$) if and only if $0\in \partial \Phi(x^{*})$. Furthermore, a function $\Phi$ is called \emph{$\sigma$-strongly convex} if for every $x,y\in \R^{d}$ and for every $p\in \partial \Phi(x)$
\begin{equation*}
  \Phi(y) \ge \Phi(x) + \langle p, y - x \rangle + \frac{\sigma}{2}\Vert x - y \Vert^{2}.
\end{equation*}
The following result is straightforward. It will play an important role in the proof of Lemma~\ref{lem:decrease}.
\begin{lemma}[Quadratic decay]%
  \label{lem:strongly-convex-function-values}
  Let $\Phi:\H \to \overline{\R}$ be a $\sigma$-strongly convex function and let $x^{*}$ denote its global minimizer. Then, it holds
  \begin{equation*}
    \Phi(x) - \Phi(x^*) \ge \frac{\sigma}{2}\lVert x-x^* \rVert^2, \quad \forall x \in \H.
  \end{equation*}
\end{lemma}
Whereas the previous statement gives a quadratic lower bound, the next one will give a quadratic upper bound. For a proof we refer to~\cite[Theorem~18.15]{bc} or~\cite{Bertsekas}.
\begin{lemma}[Descent Lemma]
  \label{lem:descent-lemma}
  Let $f:\R^{d}\!\to\!\R$ be a differentiable function with an $L$-Lipschitz gradient. Then for all $x,y \in \R^{d}$
  \begin{equation*}
	f(y) \le f(x) + \langle \nabla f(x), y-x \rangle + \frac{L}{2} \Vert y-x \Vert^2.
  \end{equation*}
\end{lemma}
The following lemma is the core of our main result. It will give an estimation for the decrease of the objective function considered in~\eqref{eq:convex_splitting}, when applying the proximal-gradient operator $T_{\alpha}$ defined in~\eqref{eq:prox-grad-op}.
For the needs of the next lemma we denote
\begin{equation}
  \label{eq:x-plus}
  x^{+} = T_{\alpha}(x), \quad \forall x \in \R^{d}.
\end{equation}
\begin{lemma}%
  \label{lem:decrease}
  Fix $x \in \R^{d}$.
  If the stepsize $\alpha>0$ is smaller than the inverse of the Lipschitz constant, i.e. $\alpha\le 1/L$, then for all $z\in \R^{d}$
  \begin{equation*}
	(g+f)(x^{+}) + \frac{1}{2 \alpha} \lVert x^{+}-z \rVert^2 \le
	(g+f)(z) + \frac{1}{2 \alpha} \lVert x - z \rVert^2.
  \end{equation*}
\end{lemma}
\begin{proof}
  First note that~\eqref{eq:x-plus} is equivalent to
  \begin{equation*}
    x^{+} = \argmin_{z\in\H{}} \left\{ g(z) + f(x) + \langle \nabla f(x),z-x \rangle + \frac{1}{2 \alpha} \lVert z - x \rVert^2 \right\}.
  \end{equation*}
  We define for all $z\in \R^{d}$
  \begin{equation*}
	l_{x}(z) := f(x) + \langle \nabla f(x), z-x\rangle
  \end{equation*}
  and
  \begin{equation*}
	\Phi_{x}(z) = g(z) + l_{x}(z) + \frac{1}{2 \alpha} \Vert z - x \Vert^2.
  \end{equation*}
  Notice that $\Phi_{x}$ is $\frac1\alpha$-strongly convex.
  Thus, by applying Lemma~\ref{lem:strongly-convex-function-values}, we have that for all $z\in\H$
  \begin{equation*}
    \Phi_{x}(x^{+}) + \frac{1}{2 \alpha} \lVert x^{+}-z \rVert^2 \le \Phi_{x}(z).
  \end{equation*}
  By the gradient inequality we know that
  \begin{equation*}
    l_{x}(z) = f(x) + \left\langle \nabla f(x), z - x \right\rangle \le f(z), \quad \forall z \in \H{}.
  \end{equation*}
  At the same time, by Lemma~\ref{lem:decrease} (Descent Lemma) and the fact that $1/\alpha\ge L$ we have that for all $z \in \R^d$
  \begin{equation}
    \label{eq:descent_lemma}
    f(x^{+}) \le l_{x}(z) + \frac{1}{2 \alpha}\lVert x^{+} - x \rVert^2
  \end{equation}
  which in return shows the statement of the lemma.
\end{proof}

\bigskip

\begin{proof}[Proof of Theorem~\ref{thm:proxgrad_selfcontracted}]
  Let $\{x_k\}_{k\in \N}$ be the sequence obtained by applying Lemma~\ref{lem:decrease} with $x:=x_{k}$, $x^{+}=x_{k+1}$ and $\alpha:=\alpha_{k}$ we get that
  \begin{equation*}
     (g+f)(x_{k+1}) + \frac{1}{2 \alpha_{k}} \lVert x_{k+1} - z \rVert^2 \le (g+f)(z) + \frac{1}{2 \alpha_{k}} \lVert x_{k} - z \rVert^2, \quad \forall z\in\H.
  \end{equation*}
  Setting $z=x_{k}$, the above yields that
  \begin{equation*}
	(g+f)(x_{k+1}) \le (g+f)(x_{k}), \quad \forall k \in \N.
  \end{equation*}
  Moreover, taking any $z\in\R^d$ such that $(g+f)(z)\le (g+f)(x_{k+1})$ we deduce
  \begin{equation*}
    \lVert x_{k+1} - z \rVert \le \lVert x_{k} - z \rVert.
  \end{equation*}
  In particular, for all $m > k+1$, we get
  \begin{equation*}
	\Vert x_{k+1} - x_{m} \Vert \le \Vert x_{k}- x_{m} \Vert.
  \end{equation*}
  Since $k$ is arbitrary in the above inequality, it can be replaced by $k+1$, yielding
  \begin{equation*}
	\Vert x_{k+2} - x_{m} \Vert \le \Vert x_{k+1}- x_{m} \Vert \le \Vert x_{k}- x_{m} \Vert.
  \end{equation*}
  Using this iterative argument for $l\in \{k+1, k+2, \dots, m\}$, we deduce
  \begin{equation*}
	\Vert x_{l} - x_{m} \Vert \le \Vert x_{k}- x_{m} \Vert.
  \end{equation*}
  This shows that the sequence is self-contracted, as asserted.
\end{proof}

\subsection{Stepsize determined via Backtracking}%
\label{sub:proximal_gradient_with_backtracking}

In practice, the Lipschitz constant of the gradient of $f$ is not always known and estimating it might lead to poor stepsizes and thus to slow convergence. In this case it is natural to use some kind of line search procedure to determine an appropriate stepsize.
We will describe one (reminiscent of Armijo test) as presented in~\cite{fista}.
The idea is the following: We want to apply Algorithm~\ref{alg:proxgrad} without the restriction $1/\alpha\ge L$ on the stepsize. In every iteration we start with an initial stepsize and decrease it until the statement of the Descent Lemma~\ref{lem:descent-lemma} is fulfilled, see below:

\begin{algo}[Proximal-Gradient with Backtracking Line search]%
  \label{alg:prox_grad_backtracking}
  For $x_{0} \in \H$, $\alpha>0$ and $0<q<1$ consider
  \begin{equation*}
  (\forall k \geq 0) \quad
    \left\lfloor \begin{array}{l l}
				   \text{Set } \alpha_{k}=\alpha \smallskip \\
      \text{while} \quad f(T_{\alpha}(x_{k})) > f(x_{k}) + \left\langle \nabla f(x_{k}), T_{\alpha}(x_{k})-x_{k} \right\rangle + \frac{1}{2 \alpha}\lVert T_{\alpha}(x_{k}) - x_{k} \rVert^2,\\ \text{do}
        \qquad  \alpha_{k} := q \alpha_{k}\smallskip\\
         x_{k+1} = T_{\alpha_{k}}(x_{k}).
    \end{array}\right.
  \end{equation*}
\end{algo}
Note that the parameter $\alpha>0$ that will be finally chosen in each iteration might be larger than $1/L$. This means that, for the cost of some extra function evaluations, a precise knowledge of the Lipschitz constant of the gradient is no more required; in addition, the algorithm might produce larger steps than what would have been allowed in Algorithm~\ref{alg:alternating-projections}, yielding a faster convergence.

\begin{theorem}%
  \label{thm:proxgrad_backtracking}
  The iterates generated by the Proximal-Gradient-Method with Backtracking Algorithm~\ref{alg:prox_grad_backtracking} form a self-contracted sequence.
\end{theorem}
\begin{proof}
  The proof follows the same lines as the one of Theorem~\ref{thm:proxgrad_selfcontracted}, the only difference being in~\eqref{eq:descent_lemma} of Lemma~\ref{lem:decrease}. This equation now holds true because of the way the iterations are chosen ensuring
  \begin{equation*}
	f(x_{k+1}) \le f(x_{k}) + \left\langle \nabla f(x_{k}), x_{k+1}-x_{k} \right\rangle + \frac{1}{2 \alpha}\lVert x_{k+1} - x_{k} \rVert^2.
  \end{equation*}
  (In the proof of Theorem~\ref{thm:proxgrad_selfcontracted} the above estimate was provided by Lemma~\ref{lem:descent-lemma}.)
\end{proof}

\subsection{Special cases: proximal-point algorithm, gradient descent}%

In problem~\eqref{eq:convex_splitting} we may consider separately the particular instances $f=0$ and $g=0$.
In the first case, the problem reduces to the minimization of a lower semicontinuous, convex function $g$ via the proximal-point algorithm. In particular, from Algorithm~\ref{alg:proxgrad} and the previous analysis, we deduce the following result, which first appeared (with a different proof) in~\cite[Theorem~4.17]{DDDL2015}
\begin{corollary}[Proximal-Point Algorithm]
  \label{cor:prox}
  Let $g: \R^{d} \to \overline{\R}$ be a convex, lower semicontinuous function and ${\{\alpha_k\}}_{k\in \N} \subseteq (0, +\infty)$. Then, for any $x_{0}\in \R^{d}$ the proximal sequence
  $$
  x_{k+1} = \prox{\alpha_{k}g}{x_{k}}, \quad \forall k \ge 0,
  $$
  is a self-contracted curve.
\end{corollary}

If $g=0$, the problem reduces to minimizing a smooth convex function with Lipschitz gradient via steepest descent. In particular, we obtain the following result, which is new. (While preparing the manuscript, the recent interesting preprint \cite{gupta2019path} came to our attention. The forthcoming result also appears there with a different proof (see~\cite[Lemma~3.1]{gupta2019path}).

\begin{corollary}[Steepest Descent]
  \label{cor:gradient-descent}
  Let $f:\R^{d} \to \R$ be a smooth convex function with $L$-Lipschitz gradient and $\{\alpha_k\}_{k\in \N}$ either be bounded from above by $1/L$ or produced by backtracking line search. Then, for any $x_{0} \in \R^{d}$, the sequence $\{x_{k}\}_{k\in \N}$ defined by
  $$
  x_{k+1} = x_{k} - \alpha_{k}\nabla f(x_{k}), \quad \forall k \ge 0,
  $$
  is self-contracted.
\end{corollary}

\subsection{A priory estimates for convergence}%
\label{sub:estimates}

An important consequence of Theorem~\ref{thm:proxgrad_selfcontracted}, Theorem~\ref{thm:proxgrad_backtracking} and Corollaries~\ref{cor:prox}--\ref{cor:gradient-descent}, is the following.
If the optimization problem~\eqref{eq:convex_splitting} has a solution, then the iterates form a bounded, self-contracted sequence. Therefore, by~\cite[Theorem~3.3]{DDDL2015} we deduce that the sequence of iterates $\{x_k\}_{k\in\N}$ has finite length, i.e.
\begin{equation*}
  \sum_{k=0}^{\infty}  \Vert x_{k+1} - x_{k} \Vert < + \infty\,,
\end{equation*}
and thus it converges to some point $x_{\infty}\in \R^d$. In addition, provided that the stepsize ${\{\alpha_k\}}_{k\in \N}$ does not go to zero too fast (for instance if it is not in $\ell^{1}(\N)$), we deduce that $x_{\infty}$ will be a minimizer of our objective function. \smallskip

Moreover, the bound on the length of the sequence of the iterates
\begin{equation*}
  \sum_{k=0}^{\infty}  \Vert x_{k+1} - x_{k} \Vert < \, C_{d} \, d_{\mathcal{S}}(x_{0}), \quad \mathcal{S}=\argmin{(f+g)}
\end{equation*}
does not depend on the data $f,g$ but only on the dimension and the distance of the initial point to the set of minimizers.

\section{Projection Algorithms}%
\label{sec:projection_algorithms}

In this section we investigate the property of self-contractedness for various projection type algorithms. Although a direct approach for establishing this property would be quite involved, it turns out that the overall analysis simplifies significantly by utilizing the main result of the previous section.\smallskip

Let us introduce our main problem of finding the intersection of two closed convex sets $A,B\subset \H$
\begin{equation}\label{eq:vlad}
  \text{Find} \quad x \in  A \cap B.
\end{equation}

\subsection{Alternating Projections}%
\label{sub:alternating_projections}

The arguably best known algorithm for solving this problem is given by
\begin{algo}[Alternating Projections{\cite[page 186]{bauschke1993convergence}}]%
  \label{alg:alternating-projections}
  For $x_{0} \in \H$, consider the iterative scheme
  \begin{equation*}
    (\forall k \geq 0) \quad
    \left\lfloor \begin{array}{l}
        y_{k+1} = P_A(x_{k}) \smallskip \\
        x_{k+1} = P_B(y_{k+1}).
    \end{array}\right.
  \end{equation*}
\end{algo}
We are going to prove that both sequences above are self-contracted. This will follow from the self-contractedness of the iterates of the Proximal-Gradient-Method.
\begin{theorem}%
  \label{thm:AP-selfcontracted}
  Let ${\{x_{k}\}}_{k\in\N}$ and ${\{y_{k}\}}_{k\in\N}$ be the two sequences of iterates generated by Algorithm~\ref{alg:alternating-projections}. Then, both sequences are self-contracted.
\end{theorem}
\begin{proof}
  We interpret alternating projections as a proximal gradient scheme.
  Define $f:= \frac12 d_A^{2}$ and note that a gradient step with respect to this function corresponds to a Projection onto $A$, i.e.\ $\Id-\nabla f= P_A$. Furthermore, we define $g:=\delta_B$ as the indicator function of the set $B$.
  Thus,
  \begin{equation*}
    x_{k+1} = P_B(P_A(x_{k})) = P_B(x_{k} - \nabla f(x_{k})) = \prox{g}{x_{k} - \nabla f(x_{k})}.
  \end{equation*}
  Therefore, Theorem~\ref{thm:proxgrad_selfcontracted} shows that ${\{x_{k}\}}_{k\in\N}$ is self-contracted, whereas self-contractedness of ${\{y_{k}\}}_{k\in\N}$ follows from symmetry.
\end{proof}

\subsection{Averaged Projections}%
\label{sub:averaged_projections}

Consider the more general problem of finding the intersection of a finite number of closed convex sets ${\{C_{i}\}}_{i=1}^n$:
\begin{equation}
  \label{eq:multiple-sets}
  \text{Find}\quad x \in \bigcap_{i=1}^n C_i.
\end{equation}
One could clearly extend the method of alternating projections to this setting and end up with the cyclic projection method, which is given by
\begin{equation*}
  x_{k+1} =  P_n\circ P_{n-1}\circ \cdots \circ P_1 (x_{k}), \quad \forall k\ge0.
\end{equation*}
In practice, this method is often replaced by other schemes (see for instance~\cite{BCC2012} and references therein). We will focus on the following modification (cf. \cite[page~368]{bauschke1996projection}).
\begin{algo}[Averaged Projections]%
  \label{alg:averaged-projection}
  For $x_0\in\H$ consider the iterative scheme
  \begin{equation}
    \label{eq:averaged-projection}
     x_{k+1} = \frac1n \sum_{i=1}^{n} P_{C_i}(x_{k}), \quad \forall k \ge 0.
  \end{equation}
\end{algo}

\begin{proposition}%
  \label{prop:averaged-projection}
  The iterates generated by Algorithm~\ref{alg:averaged-projection} form a self-contracted sequence.
\end{proposition}
We will give to different proofs for the assertion above. The first one relies on reformulating the method of averaged projections as gradient descent (with fixed stepsize), whereas the second one is based on the interpretaton of this method as alternating projections over two closed convex sets in a product space.
\begin{proof}
  \textit{(first proof)}
  Let us define for $i \in \{1,2,\dots, n\}$
  \begin{equation*}
    f_i := \frac12 d_{C_i}^{2}
  \end{equation*}
  and notice that $f:= \sum_{i=1}^{n} f_{i}$ is convex and smooth with $L$-Lipschitz gradient, where $L\le n$.
  Then, the sequence of~\eqref{eq:averaged-projection} can be equivalently defined by
  \begin{equation*}
    x_{k+1} = \left( \Id - \frac1n \nabla \bigg(\sum_{i=1}^{n} f_i \bigg) \right)(x_{k}),
  \end{equation*}
  with stepsize $\alpha_{k} = \alpha = 1/n \le 1/L$.
  The result follows by applying Corollary~\ref{cor:gradient-descent}.
  \bigskip

  \noindent\textit{(second proof)}
  Define the closed convex sets $\widehat{C}=\Pi_{i=1}^n C_i$ and $$\Delta=\{(y^{1}, y^{2}, \dots, y^{n}) \in \R^{d \times n} : y^{1}=y^{2}=\cdots=y^{n}\}.$$
  Then,~\eqref{eq:averaged-projection} is equivalent to
  \begin{equation*}
    x_{k+1} = P_\Delta(P_{\widehat{C}}(x_{k})),
  \end{equation*}
  i.e. the alternating projections method applied to the sets $\widehat{C}$ and $\Delta$ in $\R^{d\times n}$.
  Now can apply Theorem~\ref{thm:AP-selfcontracted}, and deduce the fact that ${(x_{k})}_{k\in\N}$ is self-contracted.
\end{proof}

To summarize, Algorithm~\ref{alg:alternating-projections} applied to problem~\eqref{eq:vlad} or Algorithm~\ref{alg:averaged-projection} applied to problem~\eqref{eq:multiple-sets} share the following common feature: whenever the problem is feasible or at least one of the involved sets is bounded, then the sequence of iterates is bounded. In both cases, the sequence $\{x_k\}_{k\in\N}$ is convergent and the estimates mentioned in Section~\ref{sec:preliminaries} hold true.
\medskip

\paragraph{General Conclusion.} The results of this work confirm the previous understanding~\cite{DLS2010,DDDL2015,DL2018,LMV2015} that self-contractedness relates to convexity, in both, the continuous and discrete setting. Let us mention that this concept has recently been relaxed~in~\cite{DDD2018} to so-called $\lambda$-curves (respectively $\lambda$-sequences). In that work, estimates, similar to the ones of self-contracted curves, have been obtained. However, the general asymptotic behavior of such curves remains unclear. At the same time this notion might be related to more complex settings such as accelerated convex methods or algorithms in nonconvex optimization.

\bigskip

\paragraph{Acknowledgments.} A major part of this work was done during a research
visit of the first author to the University of Chile (March to June 2019) and of
the second author to the University of Vienna (March 2020). These authors wish
to thank their hosts for hospitality.

\bigskip

\noindent Axel B\"OHM \medskip

\noindent Faculty of Mathematics, University of Vienna, \\
Oskar-Morgenstern-Platz 1, 1090 Vienna, Austria
\smallskip

\noindent E-mail: \texttt{axel.boehm@univie.ac.at}\newline%
\noindent\texttt{https://vgsco.univie.ac.at/people/phd-students/axel-boehm/}
\smallskip

\noindent Research supported by the doctoral programme VGSCO (Vienna Graduate School on Computational Optimization),
	  FWF (Austrian Science Fund), project W 1260.

\vspace{0.5cm}

\noindent Aris DANIILIDIS\medskip

\noindent DIM--CMM, UMI CNRS 2807\newline Beauchef 851, FCFM, Universidad de
Chile \smallskip

\noindent E-mail: \texttt{arisd@dim.uchile.cl} \newline
\texttt{http://www.dim.uchile.cl/~arisd/}
\smallskip

\noindent Research supported by the grants: \newline CMM AFB170001, FONDECYT
1171854 (Chile),\\ PGC2018-097960-B-C22 (Spain and EU).

\end{document}